\documentclass[a4paper,12pt]{article}

\usepackage{amsmath}
\usepackage{amsfonts}
\usepackage{amssymb}
\usepackage{amsthm}
\usepackage{amsopn}
\usepackage{amscd}
\usepackage{euscript}
\usepackage{verbatim}
\usepackage[all,cmtip]{xy}
\usepackage{pstricks}

\usepackage[utf8]{inputenc}
\usepackage{courier}
\usepackage{url}
\usepackage{color}
\usepackage{subfigure}
\usepackage{xypic}
\usepackage{graphicx}

\usepackage{longtable}
\usepackage{pstricks}
\newcommand{\scr}[1]{\ensuremath{\mathcal{#1}}}
\newcommand{\bb}[1]{\ensuremath{\mathbb{#1}}}
\usepackage{enumerate}

\usepackage[vcentering,dvips]{geometry}
\def\address#1#2{\begingroup
\noindent\parbox[t]{7.8cm}{%
\small{\scshape\ignorespaces#1}\par\vskip1ex
\noindent\small{\itshape E-mail address}%
\/: #2\par\vskip4ex}\hfill%
\endgroup}%

\author{Nathan Owen Ilten, Jacob Lewis, \& Victor Przyjalkowski}
\title{Toric Degenerations of Fano Threefolds Giving Weak Landau--Ginzburg Models}
\date{}
\newcommand{\CC}{\mathbb{C}}
\newcommand{\ZZ}{\mathbb{Z}}
\newcommand{\QQ}{\mathbb{Q}}

\newcommand{\mcH}{\mathcal{H}}

\newtheorem{lemma}{Lemma}[section]
\newtheorem{prop}[lemma]{Proposition}

\newtheorem{thm}[lemma]{Theorem}

\newtheorem{conj}[lemma]{Conjecture}
\newtheorem*{fmainthm*}{First Main Theorem}
\newtheorem*{smainthm*}{Second Main Theorem}

\theoremstyle{definition}
\newtheorem{ex}[lemma]{Example}
\newtheorem{remark}[lemma]{Remark}
\newtheorem{defn}[lemma]{Definition}

\newcommand{\NN}{{\mathbb N}}

\newcommand{\PP}{{\mathbb P}}

\newcommand{\Aff}{{\mathbb A}}

\newcommand{\conv}{\mathrm{conv}\,}
\newcommand{\init}{\mathrm{init}\,}
\newcommand{\proj}{\mathrm{Proj}\,}
\newcommand{\itc}[1]{\textup{#1}}

\newcommand{\pic}{\mathrm{Pic}\,}

\newcommand{\Spec}{\mathrm{Spec}\,}

\newtheorem{myremark}[lemma]{Remark}
\newtheorem{mydefinition}[lemma]{Definition}

\begin{document}
\maketitle
\begin{abstract}
We show that every Picard rank one smooth Fano threefold has a weak Landau--Ginzburg model coming from a toric degeneration.
The fibers of these Landau--Ginzburg models can be compactified to K3 surfaces with Picard lattice of rank 19.
We also show that any smooth Fano variety of arbitrary dimension which is a complete intersection of Cartier divisors in weighted projective space has a very weak Landau--Ginzburg model coming from a toric degeneration.
\end{abstract}



\section*{Introduction}
One of the many interpretations of mirror symmetry conjecturally relates the quantum cohomology of a smooth Fano variety $X$ to the Picard--Fuchs operator of a pencil $f:Y\to\CC$ called a Landau--Ginzburg model for $X$. Given some Fano $X$, it is not clear whether such a Landau--Ginzburg model exists, how to find one assuming the existence, or what additional assumptions one should make on such $f$ to ensure  uniqueness.

Given a pencil $Y\to\CC$, passing to certain open subsets of $Y$ will preserve the part of the Picard--Fuchs operator relevant to mirror symmetry.
In \cite[Conjecture 36]{Prz09}, the second author conjectured that one can always find a Landau--Ginzburg model of the form $f:Y\to\CC$, where $Y=(\CC^*)^n$ is a torus of dimension equal to that of $X$. In this case, $f$ can be represented by a Laurent polynomial in $n$ variables. The underlying motivation is that if $X$ degenerates to some ``nice'' toric variety with moment polytope $\nabla$, the quantum cohomology of $X$ should be related to the Picard--Fuchs operator for a Laurent polynomial whose Newton polytope is dual to $\nabla$. Thus, this conjecture motivates the question concerning  to which toric varieties a given Fano $X$ degenerates.

Since smooth Fano threefolds have been completely classified, see \cite{Isk77}, \cite{Isk78}, and \cite{mori:81a}, they provide a good testing ground for this conjecture. Indeed, in \cite{Prz09}, the second author has shown that for all smooth Fano threefolds of Picard rank one, there is a Laurent polynomial giving a weak Landau--Ginzburg model, see Subsection \ref{subsec:ms} for a precise definition. The first main result of this present article is to show that these Laurent polynomials do in fact come from toric degenerations of the corresponding Fano varieties:

\begin{fmainthm*}[Theorem \ref{theorem:Main}]
Each smooth Fano threefold of Picard rank 1 has a weak Landau--Ginzburg model associated with a toric degeneration.
More precisely, the Laurent polynomials in Table \ref{table} are weak Landau--Ginzburg models for corresponding Fano varieties. For each polynomial $f$ in the table, the corresponding Fano degenerates to the toric variety with
moment polytope dual to the Newton polytope of $f$.
\end{fmainthm*}

We construct these toric degenerations via a number of techniques.  For Fano complete intersections in weighted projective spaces, we show the existence of a very weak Landau--Ginzburg model with corresponding toric degeneration in arbitrary dimension, see Theorem \ref{thm:ci}. The essential ingredient here is K.\,Altmann's construction of toric deformations, \cite{altmann:95a}.  For Picard rank one Fano threefolds, we deal with the remaining cases by using techniques of monomial degenerations \cite{ilten:11a} and
previously known small toric degenerations \cite{galkin:08a}.
For additional techniques in constructing toric degenerations not applied here, see \cite{alexeev:04a}
and \cite{kapustka:11a}.

The fibers of the Landau--Ginzburg models we consider can be compactified to K3 surfaces as shown in \cite{Prz09}. In the present paper, we show that the Picard lattices of these surfaces all have the expected rank:
\begin{smainthm*}[Theorem \ref{appendixthm}]
 Let $X$ be a Fano threefold of Picard number one, and $f$ the Laurent polynomial for $X$ in Table 1.  Then the fibers of $f$ compactify to a family of K3 surfaces of Picard rank 19.
\end{smainthm*}

Recently T.\,Coates, A.\,Corti, S.\,Galkin, V.\ Golyshev, A.\,Kasprzyk et al.  have made progress in computing $I$-series and very weak Landau--Ginzburg
models for all smooth Fano threefolds of any rank (see~\cite{fanosearch}).
Some of them are known to be given by toric degenerations. The natural problem is to generalize this paper to all Fano threefolds
using their work. J.\,Christophersen and N.\,Ilten have recently classified all embedded degenerations of smooth Fano threefolds of degree at most twelve to toric Fano varieties with Gorenstein singularities (see~\cite{ilten:12a}).
Also, V.\,Batyrev and M.\,Kreuzer 
have recently constructed degenerations of rank one Fano threefolds to complete intersections in toric varieties (see~\cite{BK12}).

A recent idea of L.\,Katzarkov  is to relate the vanishing cycles of the central fibers of compactified weak Landau--Ginzburg models
for Fano varieties with birational
invariants of these varieties.
In a series of papers (
\cite{Prz09}, \cite{IKP11}, 
\cite{KP11}, \cite{CKPa}, \cite{CKPb}, \cite{DKLP}), this idea is applied to the weak Landau--Ginzburg models discussed in this paper to study certain invariants of Fano threefolds
(Hodge type, rationality, etc.).

This article is organized as follows. In Section \ref{sec:prelim} we introduce notation and necessary definitions, first dealing with polytopes and toric varieties, and then with Landau--Ginzburg models. We then introduce our techniques of toric degeneration in Section \ref{sec:degen}; in particular, Section \ref{subsec:ci} contains our result regarding toric degenerations of Fano complete intersections. In Section \ref{sec:main} we then collect everything together to prove the first main theorem.
Section \ref{sec:picrk} then contains the discussion of the Picard lattices for the compactified fibers of our Landau--Ginzburg models.

\medskip

{\bf Acknowledgments.}
The authors are grateful to L.\,Katzarkov for helpful comments.

N.\,I. was supported by the Max-Planck-Institut f\"ur Mathematik. V.\,P. was partially supported by NSF FRG DMS-0854977, NSF DMS-0854977, NSF DMS-0901330, grants FWF P 24572-N25 and FWF P20778,
RFFI grants 11-01-00336-a, 11-01-00185-a, and 12-01-31012, grants MK$-1192.2012.1$,
NSh$-5139.2012.1$, and AG Laboratory GU-HSE, RF government
grant, ag. 11 11.G34.31.0023.
J.\,L. was supported by NSF grant OISE-0965183.
The paper was partially written during the first author's visit to Moscow with support from the  Dynasty Foundation.

\newpage

\begin{center}
\begin{longtable}{||c|c|c|p{5cm}|p{6cm}||}
  \hline
  No. & Index & Degree &
\begin{minipage}[c]{5cm}
\centering
\vspace{.1cm}

  Description

\vspace{.1cm}

\end{minipage}
  &
\begin{minipage}[c]{5cm}
\centering
\vspace{.1cm}

  Weak LG model

\vspace{.1cm}

\end{minipage}
\\
  \hline
  \hline
  1 & 1 & 2 &
\begin{minipage}[c]{5cm}
\vspace{.1cm}

Sextic double solid $X_2$ (double cover of $\PP^3$ ramified over smooth sextic).

\vspace{.1cm}

\end{minipage}
&
\begin{minipage}[c]{6cm}
\vspace{.1cm}
\centering

$\frac{(x+y+z+1)^6}{xyz}$

\vspace{.1cm}

\end{minipage}
    \\
  \hline
  2 & 1 & 4 &
\begin{minipage}[c]{5cm}
\vspace{.1cm}

  The general element of the family is quartic $X_4$.

\vspace{.1cm}

\end{minipage}
&
\begin{minipage}[c]{6cm}
\vspace{.1cm}
\centering

  $\frac{(x+y+z+1)^4}{xyz}$

\vspace{.1cm}

\end{minipage}
  \\
  \hline
  3 & 1 & 6 &
\begin{minipage}[c]{5cm}
\vspace{.1cm}

Smooth complete intersection of quadric and cubic $X_6$.

\vspace{.1cm}

\end{minipage}
&
\begin{minipage}[c]{6cm}
\vspace{.1cm}
\centering

$\frac{(x+1)^2(y+z+1)^3}{xyz}$

\vspace{.1cm}

\end{minipage}
\\
  \hline
  4 & 1 & 8 &
\begin{minipage}[c]{5cm}
\vspace{.1cm}

Smooth complete intersection of three quadrics $X_8$.

\vspace{.1cm}

\end{minipage}
&
\begin{minipage}[c]{6cm}
\vspace{.1cm}
\centering

$\frac{(x+1)^2(y+1)^2(z+1)^2}{xyz}$

\vspace{.1cm}

\end{minipage}
\\
  \hline
  5 & 1 & 10 &
\begin{minipage}[c]{5cm}
\vspace{.1cm}

  The general element is $X_{10}$, a section of $G(2,5)$ by 2 hyperplanes in Pl\"{u}cker embedding and quadric.


\end{minipage}

&
\begin{minipage}[c]{6cm}
\vspace{.1cm}
\centering

$\frac{(1+x+y+z+xy+xz+yz)^2}{xyz}$


\end{minipage}
\\
  \hline
  6 & 1 & 12 &
\begin{minipage}[c]{5cm}
\vspace{.1cm}

  Variety $X_{12}$.

\vspace{.1cm}

\end{minipage}
&
\begin{minipage}[c]{6cm}
\vspace{.1cm}
\centering

  $\frac{(x+z+1)(x+y+z+1)(z+1)(y+z)}{xyz}$

\vspace{.1cm}

\end{minipage}
\\
  \hline
  7 & 1 & 14 &
\begin{minipage}[c]{5cm}
\vspace{.1cm}

  Variety $X_{14}$, a section of $G(2,6)$ by 5 hyperplanes in Pl\"{u}cker embedding.

\vspace{.1cm}

\end{minipage}
  &

\begin{minipage}[c]{6cm}
\vspace{.1cm}
\centering

  $\frac{(x+y+z+1)^2}{x}$

  $+\frac{(x+y+z+1)(y+z+1)(z+1)^2}{xyz}$

\vspace{.1cm}

\end{minipage}
\\
  \hline
  8 & 1 & 16 &
\begin{minipage}[c]{5cm}
\vspace{.1cm}

  Variety $X_{16}$.

  \vspace{.1cm}

\end{minipage}
  &
  \begin{minipage}[c]{6cm}
\vspace{.1cm}
\centering

$\frac{(x+y+z+1)(x+1)(y+1)(z+1)}{xyz}$

\vspace{.1cm}

\end{minipage}
\\
  \hline
  9 & 1 & 18 &
\begin{minipage}[c]{5cm}
\vspace{.1cm}

  Variety $X_{18}$.

\vspace{.1cm}

\end{minipage}
  &
\begin{minipage}[c]{6cm}
\vspace{.1cm}
\centering

$\frac{(x+y+z)(x+xz+xy+xyz+z+y+yz)}{xyz}$
\vspace{.1cm}

\end{minipage}
\\
  \hline
  10 & 1 & 22 &
\begin{minipage}[c]{5cm}
\vspace{.1cm}

  Variety $X_{22}$.

\vspace{.1cm}

\end{minipage}
  &
\begin{minipage}[c]{6cm}
\vspace{.1cm}
\centering


$\frac{(z + 1) (x + y + 1) (x y + z)}{x y z} + \frac{x y}{z} + z + 3$

\vspace{.1cm}

\end{minipage}
\\
  \hline
  11 & 2 & $8\cdot 1$ &
\begin{minipage}[c]{5cm}
\vspace{.1cm}

  Double Veronese cone $V_{1}$
  (double cover of the cone over the Veronese surface branched in a smooth cubic).

\vspace{.1cm}

\end{minipage}
&
\begin{minipage}[c]{6cm}
\vspace{.1cm}
\centering

  $\frac{(x+y+1)^6}{xy^2z}+z$

\vspace{.1cm}

\end{minipage}
\\
  \hline
  12 & 2 & $8\cdot 2$ &
  \begin{minipage}[c]{5cm}
\vspace{.1cm}

Quartic double solid $V_2$ (double cover of $\PP^3$ ramified over smooth quartic).

\vspace{.1cm}

\end{minipage}
&
\begin{minipage}[c]{6cm}
\vspace{.1cm}
\centering

$\frac{(x+y+1)^4}{xyz}+z$

\vspace{.1cm}

\end{minipage}
\\
  \hline
  13 & 2 & $8\cdot 3$ &
  \begin{minipage}[c]{5cm}
\vspace{.1cm}

Smooth cubic $V_3$.

\vspace{.1cm}

\end{minipage}
&
\begin{minipage}[c]{6cm}
\vspace{.1cm}
\centering

$\frac{(x+y+1)^3}{xyz}+z$

\vspace{.1cm}

\end{minipage}
\\
  \hline
  14 & 2 & $8\cdot 4$ &
\begin{minipage}[c]{5cm}
\vspace{.1cm}

  Smooth intersection of two quadrics $V_4$.

\vspace{.1cm}

\end{minipage}
&
\begin{minipage}[c]{6cm}
\vspace{.1cm}
\centering

  $\frac{(x+1)^2(y+1)^2}{xyz}+z$

\vspace{.1cm}

\end{minipage}
\\
  \hline
  15 & 2 & $8\cdot 5$ &
\begin{minipage}[c]{5cm}
\vspace{.1cm}

  Variety $V_{5}$, a section of $G(2,5)$ by 3 hyperplanes in Pl\"{u}cker embedding.

\vspace{.1cm}

\end{minipage}
  &
\begin{minipage}[c]{6cm}
\vspace{.1cm}
\centering

  $x+y+z+\frac{1}{x}+\frac{1}{y}+\frac{1}{z}+xyz$

\vspace{.1cm}

\end{minipage}
\\
  \hline
  16 & 3 & $27\cdot 2$ &
\begin{minipage}[c]{5cm}
\vspace{.1cm}

  Smooth quadric $Q$.

\vspace{.1cm}

\end{minipage}
&
\begin{minipage}[c]{6cm}
\vspace{.1cm}
\centering

  $\frac{(x+1)^2}{xyz}+y+z$

\vspace{.1cm}

\end{minipage}
\\
  \hline
  17 & 4 & $64$ &
\begin{minipage}[c]{5cm}
\vspace{.1cm}

  $\PP^3$.

\vspace{.1cm}

\end{minipage}
&
\begin{minipage}[c]{6cm}
\vspace{.1cm}
\centering

  $x+y+z+\frac{1}{xyz}$

\vspace{.1cm}

\end{minipage}
\\
  \hline
  \hline

\caption[]{\label{table}Weak Landau--Ginzburg models for Fano threefolds with toric degenerations}
\end{longtable}
\end{center}

\section{Preliminaries}\label{sec:prelim}
\subsection{Polytopes and Toric Varieties}
We begin by fixing notation and introducing some basic concepts for toric varieties, see \cite{fulton:93a} for more details. Throughout the article, we will use $N$ to denote some lattice, with $M$ its dual, and $N_\QQ$, $M_\QQ$ the associated $\QQ$-vector spaces.

For any Laurent polynomial $f=\sum_{v\in N} c_v\cdot\chi^v$ in $\CC[N]$, its \emph{Newton polytope} is defined to be
$$
\Delta_f:=\conv \{v\ |\ c_v\neq 0\}.
$$

For any polytope $\Delta$ in $N_\QQ$ containing the origin in its interior, we define its \emph{dual polytope} to be
$$
\Delta^*:=\{u\in M_\QQ\ |\ \min_{v\in\Delta} \langle v,u\rangle \geq -1\}.
$$
If $\Delta$ and $\Delta^*$ are both lattice polytopes then they are called \emph{reflexive}, see \cite{batyrev:94a} for more details.

Consider now some rational polytope $\nabla\subset M_\QQ$.
This gives rise to two semigroups:
\begin{align*}
	S_\nabla:=\{ (u,k)\in M\times \NN\ |\ u\in k\cdot(\nabla\cap M)\}\\
	\widetilde{S}_\nabla:=\{ (u,k)\in M\times \NN\ |\ u\in (k\cdot\nabla)\cap M\}
\end{align*}
with $S_\nabla\subset \widetilde{S}_\nabla$.
From these semigroups, we can construct projective toric varieties
\begin{align*}
\PP(\nabla):=\proj \CC[S_\nabla],\qquad
\widetilde\PP(\nabla):=\proj \CC[\widetilde{S}_\nabla].
\end{align*}
Via this construction, $\PP(\nabla)$ is embedded in $\PP^n$ with $n=\#\nabla\cap M-1$, whereas $\widetilde{\PP}(\nabla)$ in general is only embedded in some weighted projective space. The dimension of $\PP(\nabla)$ is the dimension of the convex hull of $\nabla\cap M$, and the dimension of $\widetilde\PP(\nabla)$ is the dimension of $\nabla$.
The inclusion of semigroups induces a map $\rho:\widetilde\PP(\nabla)\to\PP(\nabla)$;  if $\nabla$ is a lattice polytope such that $\nabla\cap M$ generates the lattice $M$, then this is simply the normalization map. We say that $\nabla$ is \emph{very ample} if $\rho$ is an isomorphism. This is in particular the case if $S_\nabla=\widetilde{S}_\nabla$.
Note that if $\nabla$ is a lattice polytope admitting a unimodular triangulation, then we do in fact have this equality, i.e. $S_\nabla=\widetilde{S}_\nabla$.

Consider a lattice polytope $\Delta\in N_\QQ$ with the origin in its interior whose vertices are all primitive lattice elements, and
set $X=\widetilde{\PP}(\Delta^*)$. Then $X$ is Fano,
i.e. $-K_X$ is 
ample. If $\Delta$ is reflexive, then $X$ is even Gorenstein. Furthermore, for $k\in \NN$, $k\cdot \Delta^*$ is very ample if and only if $k(-K_X)$ is. For $k (-K_X)$ very ample, the corresponding embedding of $X$ is given by $\PP(k\cdot\Delta^*)$.
Finally, $X$ has at worst canonical singularities if and only if the sole lattice point in the interior of $\Delta$ is the origin.

\subsection{Mirror Symmetry of Variations of Hodge Structures}
\label{subsec:ms}
We state a version of the mirror symmetry conjecture of variations of Hodge structures adopted to our goals following~\cite{Prz09}.
For more details,  see loc. cit. 
and the references therein.

For any smooth Fano variety $X$ (via its Gromov--Witten invariants) one can construct the so called \emph{regularized quantum differential operator} $L_X$
(equivalently, Dubrovin's second structural connection), see for instance~\cite{Prz07}.
For \emph{quantum minimal varieties} (corresponding, in particular,
to Fano complete intersections in weighted projective spaces or Fano threefolds of Picard rank 1) they are \emph{of type DN}, see say~\cite{GS07} or~\cite{Prz08}.
Such operators were studied in~\cite{GS07}.

We define this explicitly for a Fano threefold $X$. By
$$
a_{ij}=\langle(-K_X)^i,(-K_X)^{3-j},-K_X\rangle_{j-i+1}, \ \ \ 0\leq
i\leq j\leq 3,\ j>0,
$$
we denote the Gromov--Witten invariant whose meaning is the expected number of
rational curves of anticanonical degree $j-i+1$ that intersect
general representatives of the homological classes dual to
$(-K_X)^i,(-K_X)^{3-j},-K_X$.
It turns out that such numbers determine the even part of the Gromov--Witten theory of
$X$. Moreover, the regularized quantum $\mathcal D$-module for $X$ may
be represented by a differential equation of type D3 with polynomial coefficients in the $a_{ij}$'s:

\begin{mydefinition}
\label{definition:D3} Consider the ring $\mathcal D=\CC[t,
\frac{\partial}{\partial t}]$ and  differential operator
$D=t\frac{\partial}{\partial t}\in \mathcal D$. \emph{The
regularized quantum differential operator} or \emph{operator of type
D3} associated with the Fano threefold $X$ is the operator
$$
\begin{array}{l}
L_X={D}^{3}-t \left (2\,D+1\right )\left (\lambda {D}^{2
}+(a_{{11}}+\lambda){D}^{2}+\lambda D+(a_{{11}}+\lambda)D+\lambda\right ) \\
+ {t}^{2} \left (D+1\right ) \left
(({a_{{11}}}+\lambda)^{2}{D}^{2}+{\lambda}^{2}{D}^{2}+4\,(a_{{11}}+\lambda)\lambda
{D}^{2}-a_{{12}}{D}^{2}-2\,a_
{{01}}{D}^{2} \right. \\
\left. + 8\,(a_{{11}}+\lambda)\lambda
D-2\,a_{{12}}D+2\,{\lambda}^{2}
D-4\,a_{{01}}D+2\,({a_{{11}}}+\lambda)^{2}D+6\,(a_{{11}}+\lambda)\lambda
\right. \\ \left.+{\lambda}^ {2}-4\,a_{{01}}\right ) - {t}^{3}\left
(2\, D+3\right )\left (D+2\right )\left (D+1\right )\left
({\lambda}^{2}(a_{{11}}+\lambda)+{
(a_{{11}}}+\lambda)^{2}\lambda-a_{{12}}\lambda+a_{{02}}\right.
\\ \left. -(a_{{11}}+\lambda)a_{{01} }-a_{{01}}\lambda\right ) + {t}^{4}\left
(D+3 \right )\left (D+2\right )\left (D+1\right )\left
(-{\lambda}^{2}a_{
{12}}+2\,a_{{02}}\lambda+{\lambda}^{2}({a_{{11}}}+\lambda)^{2}
\right. \\ \left. -a_{{03}}+
{a_{{01}}}^{2}-2\,a_{{01}}(a_{{11}}+\lambda)\lambda\right ),
\end{array}
$$
defined up to a shift $\lambda\in\CC$.
\end{mydefinition}

\begin{mydefinition}(The unique) analytic solution of the equation
$L_XI=0$ of type
$$
I^X_{H^0}=1+a_1t+a_2t^2+\ldots\in \CC [[t]],\ \ \ \ a_i\in \CC,
$$
is called \emph{the fundamental term of the regularized $I$-series}
of $X$.
\end{mydefinition}

According to A.\,Givental this series is the constant term (with respect to cohomology) of
the regularized $I$-series for $X$, i.\,e. of the generating series for 1-pointed
Gromov--Witten invariants (see, for instance,~\cite{Prz08}).

Consider the torus $(\CC^*)^n= \Spec \CC[\ZZ^n]$ and a regular function $f$ on it. This function may
be represented by a Laurent polynomial in the variables $x_1,\ldots,x_n$. Let $\phi_f(i)$ be the constant
term  of $f^i$. Put
$$
\Phi_f=\sum_{i=0}^\infty \phi_f(i)\cdot t^i\in \CC[[t]].
$$

\begin{mydefinition} The series $\Phi_f$ 
is
called \emph{the constant terms series} of $f$.
\end{mydefinition}

The following theorem is a sort of mathematical folklore (see, for
instance,~\cite[Proposition~2.3]{Prz07}). It states that the
constant terms series of Laurent polynomial is the main period of a
pencil given by this polynomial.

\begin{thm}
\label{Proposition: Picard--Fuchs} Consider a pencil $(\CC^*)^n\to
\Aff^1=\PP^1\setminus\{0\}$ given by the  Laurent polynomial $f\in
\CC[\ZZ^n]$ with fibers
$Y_\lambda=\{1-\lambda f=0\}$ for $\lambda\in \CC^*\subset \PP^1$ and $Y_\infty=\{f=0\}$.
Assume that the Newton polytope of $f$ contains $0$ in the interior and let
$t$ be a local coordinate at $0$. Then there is a
fiberwise $(n-1)$-form $\omega_t\in \Omega^{n-1}_{(\CC^*)^n/\Aff^1}$ and
\itc{(}locally defined\itc{)} fiberwise $(n-1)$-cycle $\Delta_t$
such that
$$
\Phi_f=\int_{\Delta_t}\omega_t.
$$
\end{thm}
\begin{mydefinition}
\label{Definition: weak LG models} Let $X$ be a
smooth 
Fano variety of dimension $n$ and $I^X_{H^0}\in \CC[[t]]$ be the fundamental term
of its regularized $I$-series.
\begin{itemize}
\item A Laurent polynomial $f\in \CC[\ZZ^n]$ is called a \emph{very weak Landau--Ginzburg model} for $X$ if $\Delta_f$ contains the origin in its interior and (up
to some constant shift $f\mapsto f+\alpha$, $\alpha\in \CC$)
$$
\Phi_f=I^X_{H^0}.
$$
\item A Laurent polynomial $f\in \CC[\ZZ^n]$ is called \emph{a weak
Landau--Ginzburg model} for $X$ if it is a very weak
Landau--Ginzburg model for $X$ 
and if it admits a \emph{Calabi--Yau compactification}, i.\,e. there is a fiberwise compactification of a family $f\colon (\CC^*)^n\to \CC$
whose total space is (an open) smooth Calabi--Yau
variety.
\end{itemize}
\end{mydefinition}

\begin{myremark}
By the above theorem, if $f$ is a very weak Landau--Ginzburg model, then $L_X=PF_f$, where $PF_f$ is the
Picard--Fuchs operator for the pencil given by $f$.
\end{myremark}

\begin{conj}[Mirror Symmetry of variations of Hodge structures]
For any Fano variety $X$ there exists a one-parameter family $Y\to \CC$ whose
Picard--Fuchs $\mathcal D$-module is isomorphic to a regularized
quantum $\mathcal D$-module for $X$.
\end{conj}

Assume that $\dim X=3$, $\pic (X)=\ZZ$, and $Y=(\CC^*)^3$. Then this conjecture
reduces to the following theorem:

\begin{thm}[{\cite[Theorem~18]{Prz09}}]
For any smooth Fano threefold $X$ with Picard number 1 there exists
a  (weak) Landau--Ginzburg model.
\end{thm}

There are 17 families of smooth Fano varieties of Picard rank 1, see~\cite{Isk77} and \cite{Isk78}.
In~\cite{Prz07} there is a list of weak Landau--Ginzburg models for all of them (our Table~\ref{table}).
There it is noted that polynomials from the list are \emph{potentially toric}, i.~e. there are in fact no Hilbert polynomial restrictions for degenerating the Fano threefolds to the toric varieties associated to the corresponding weak Landau--Ginzburg models.
In Theorem~\ref{theorem:Main} we prove that weak Landau--Ginzburg models are \emph{toric}, that is these Fano threefolds actually \emph{can} be degenerated to corresponding toric varieties.

\section{Toric Degenerations of Fano Varieties}\label{sec:degen}

\subsection{Complete Intersections in Weighted Projective Spaces}\label{subsec:ci}

Consider a smooth Fano complete intersection $X$ of Cartier divisors of degrees $n_1,\ldots, n_k$ in weighted projective space
$\PP(w_0:\ldots:w_r)$, $w_0\leq w_1\leq\ldots\leq w_r$.
Let $n_0=\sum w_i-\sum n_j$ be its Fano index.
By~\cite[Proposition~7]{Prz10}, $w_0=1$ and there is a partition of $I=\{0\ldots r\}$ into $k+1$ non-intersecting sets $I_0,\ldots,I_k$ such that
$$
\sum_{j\in I_i} w_j=n_i,\ \ \ i=0,\ldots,k,
$$
and $w_0\in I_0$ (the so called \emph{$\QQ$-nef-partition}). Let $w_{i0},\ldots,w_{im_i}$ denote the elements of $I_i$ for $0\leq i \leq k$.
By~\cite[Theorem~9]{Prz10}, there is a Hori--Vafa very weak Landau--Ginzburg model for $X$ defined by
$$
f_X=\frac{(x_{1,0}+\ldots+x_{1,m_i})^{n_1}\cdot\ldots\cdot(x_{k,0}+\ldots+x_{k,m_k})^{n_k}}{\prod  x_{i,j}^{w_{ij}}}+x_{0,1}+\ldots+x_{0,m_0},
$$
where $x_{i,0}$ is just the constant $1$ for $0\leq i \leq m_i$.

\begin{myremark}
\label{remark:good partition}
By~\cite[Proposition~7]{Prz10}, one can choose a partition $I_0,I_1,\ldots,I_k$ such that $w_{0j}=1$ for $0\leq j \leq m_0$. For the remainder of the section, we will assume that the partition has been chosen in this manner. Note that this implies that $n_0=m_0+1$.
\end{myremark}

Since $X$ is a complete intersection, we can degenerate the defining equations to sufficiently general binomials in order to attain a toric degeneration of $X$. However, we would in fact like to attain a toric degeneration to the variety corresponding to our $f_X$; this is the content of the following theorem.

\begin{thm}\label{thm:ci}
	There is a flat degeneration of $X$ to $Z=\widetilde{\PP}(\Delta_{f_X}^*)$.
\end{thm}
\begin{proof}
	To prove the theorem, we will show that $Z$ can be embedded as a complete intersection of degrees $n_1,\ldots,n_k$ in $\PP(w_0:\ldots:w_r)$. One way of doing this is by explicitly comparing generators and relations for the $d\cdot n_0$-th antipluricanonical embedding of $Z$ for some $d$ with those for the $n_0$-th Veronese embedding of $\PP(w_0:\ldots:w_r)$, see Example \ref{ex:ci} for a demonstration of this.
	Here we take a more intrinsic approach, avoiding generators and relations as much as possible. We will first perform a coordinate transformation and pass to a Veronese superalgebra to arrive at a more usable description of $Z$. We will then apply a result of K.\,Altmann (see \cite[Theorem 3.5]{altmann:95a}) which relates Minkowski decompositions of polytopes to toric complete intersections, i.e. a toric variety $X_1$ embedded equivariantly as a complete intersection in a second toric variety $X_2$. In our case, $X_1$ will just be our variety $Z$, and $X_2$ will be the desired weighted projective space.

We first describe our variety $Z$.
Consider the lattice
$$
N=\bigoplus_{i=0}^k \ZZ^{m_i}
$$
with basis $b_{ij}$ for  $0\leq i \leq k$ and $1\leq j \leq m_i$; let $M$ be the dual lattice. For any $i$, we set $b_{i0}=0$. Let $\Delta_i=\conv\{b_{ij}\}_{j=0}^{m_i}$ for $i\geq 1$, and set $\Delta_0=\conv\{b_{ij}\}_{j=1}^{m_0}$. Then we have that
$$
\Delta_{f_X}=\conv\left(\sum_{i\geq 1} n_i \Delta_i-\sum_{i \geq 0,j\geq 1} w_{ij}b_{ij},\Delta_0\right).
$$
Set $\sigma=\QQ_{\geq 0}\cdot (\Delta_f,1)$, and let $c$ be the vector $(\underline{0},1)$ in $N\oplus \ZZ$.
Then $Z=\widetilde\PP(\Delta_f^*)$ is just $\proj \CC[\sigma^\vee\cap (M\oplus \ZZ)]$, where the $\ZZ$-grading for $\proj$ is given by $c$.

We now perform a coordinate transformation to bring our description of $Z$ into more usable form.
Consider the lattice automorphism of $N\oplus \ZZ$ sending
\begin{align*}
b_{0j}&\mapsto b_{0j}-c&\qquad 1\leq j \leq m_i& \\
b_{ij}&\mapsto b_{ij}&\qquad i\geq 1, 1\leq j \leq m_i&\\
c&\mapsto c
\end{align*}
This maps sends $\sigma$ to $\sigma'$, where
$$
\sigma'=\QQ_{\geq 0} \cdot \left(\sum_{i\geq 1} n_i\Delta_i-\sum_{i \geq 0,j\geq 1} w_{ij}b_{ij}+n_0c\right)+\QQ_{\geq 0} \Delta_0.
$$
Note that we can replace $\sigma'$ by
$$
\sigma''=\QQ_{\geq 0} \cdot \left(\sum_{i\geq 1} n_i\Delta_i-\sum_{i \geq 0,j\geq 1} w_{ij}b_{ij}+c\right)+\QQ_{\geq 0} \Delta_0
$$
and we still have that $Z\cong \proj \CC[(\sigma'')^\vee\cap (M\oplus \ZZ)]$, where the $\ZZ$-grading for $\proj$ is again given by $c$.
Indeed, in the inclusion
$$ \CC[(\sigma')^\vee\cap (M\oplus \ZZ)]\hookrightarrow \CC[(\sigma'')^\vee\cap (M\oplus \ZZ)]$$
coming from the lattice inclusion $M\oplus\ZZ\hookrightarrow M\oplus \ZZ$ sending $c^*$ to $n_0c^*$, the left hand side is just the $n_0$th Veronese subalgebra of the right hand side. (We denote the elements of the basis of $M\times\ZZ$ dual to $b_{ij}$ and $c$ by respectively $b_{ij}^*$ and $c^*$.)

We now apply the result of Altmann to demonstrate $X_1=Z$ as a complete variety in another toric variety $X_2$. Let $Q$ be the intersection of $\sigma''$ with the hyperplane $[c^*=1]$, viewed via the natural cosection as living in $N_\QQ$. Explicitly, we have
$$Q=\sum_{i\geq 1} n_i\Delta_i-\sum_{i \geq 0,j\geq 1} w_{ij}b_{ij}+\QQ_{\geq 0}\Delta_0.$$
Thus, we have a natural decomposition of $Q$ into a Minkowski sum with summands whose compact parts consist  of the point $-\sum_{i \geq 0,j\geq 1} w_{ij}b_{ij}$ and dilated simplices $n_i\Delta_i$ for $i\geq 1$.
Consider the lattice $\widehat N=N\oplus \ZZ^{k+1}$, where the second component has basis $c_i$ for $0\leq i \leq k$; let $\widehat M$ be the dual lattice. Define the cone $\widehat \sigma\subset \widehat N_\QQ$ to be generated by
\begin{align*}
&\Delta_0,\qquad \Delta_i+c_i,\qquad\qquad 1\leq i \leq k.\\
-&\sum_{i \geq 0,j\geq 1} w_{ij}b_{ij}+c_0.
\end{align*}
By \cite[Theorem 3.5]{altmann:95a} there is a closed embedding
$$\proj  \CC[(\sigma'')^\vee\cap (M\oplus \ZZ)]\hookrightarrow \proj \CC[(\widehat\sigma)^\vee\cap (\widehat M)]=:X_2,$$
where the $\ZZ$-grading for the latter semigroup algebra is given by $\widehat c=c_0+\sum_{i=1}^k n_ic_i$. By the same theorem,
this embedding is given exactly by the equations
\begin{align*}\label{eqn:cieqn}
\chi^{n_ic_0^*}-\chi^{c_i^*}\qquad\qquad 1\leq i \leq k,
\end{align*}
where for $u\in\widehat{M}$, $\chi^u$ denotes the corresponding character.

We now show that $X_2$ is just the desired weighted projective space.
An explicit calculation gives that $(\widehat{\sigma})^\vee$ is generated by the vectors
\begin{align*}
&b_{ij}^*+w_{ij}c_0^*,\qquad\qquad 0\leq i \leq k,\quad 1\leq j \leq m_i,\\
&c_0^*,\\
&c_i^*-\sum_{j=1}^{m_i}(b_{ij}^*-w_{ij}c_0^*),\qquad\qquad 1\leq i \leq k,
\end{align*}
and is thus a smooth simplicial cone, where the generators have respectively weights $w_{ij}$,  $w_{00}=1$, and $w_{i0}=n_i-\sum_{j\geq 1} w_{ij}$ with respect to $\widehat c$.
Thus, $\proj \CC[(\widehat\sigma)^\vee\cap (\widehat M)]$ is the typical description of $\PP(w_{00}:\ldots:w_{km_k})$ and we have embedded $Z$ as a complete intersection of degrees $n_1,\ldots,n_k$. By degenerating the equations defining $X$ in $\PP(w_{00}:\ldots:w_{km_k})$, we get a degeneration of $X$ to $Z$.
\end{proof}

\begin{ex}[The del Pezzo surface of degree 2]\label{ex:ci}
We now consider the example of the del Pezzo surface of degree 2 to hint at an alternate approach to the above theorem via generators and relations. This is a hypersurface of degree 4 in $\PP(1,1,1,2)$. Its weak Landau--Ginzburg model presented above is thus
$$
f_X=\frac{(x+y+1)^4}{xy}.
$$
The corresponding Newton polytope $\Delta_f$ has vertices equal to the columns of the matrix
$$
\left(
  \begin{array}{rrr}
3&-1&-1\\
-1&3&-1\\
  \end{array}
\right).
$$
The dual polytope $\Delta_f^*$ thus has vertices equal to the columns of the matrix
$$
\left(
  \begin{array}{rrr}
	  1 & 0 & -1/2\\
	  0&1&-1/2\\
  \end{array}
\right).
$$
This is not a lattice polytope; in particular $Z=\widetilde\PP(\Delta_f^*)\neq \PP(\Delta_f^*)$. However, its double dilation $\nabla=2\cdot \Delta_f^*$ is in fact very ample.
The integral points of $\nabla$ are $u=(-1,-1)$ and $v_{ab}=(a,b)$ for $a,b\geq 0$, $a+b\leq 2$. These correspond to generators for the homogeneous coordinate ring of $Z$ in this (the doubleanticanonical) embedding.

Affine homogeneous relations among these lattice points correspond to binomial relations in the ideal of $Z$. In this case, these relations are generated by five 2-Veronese type relations
\begin{align*}
	&v_{20}+v_{02}=2v_{11}, &v_{20}+v_{01}=v_{10}+v_{11},\\
	&v_{20}+v_{00}=2v_{10}, &v_{02}+v_{10}=v_{01}+v_{11},\\
	&v_{02}+v_{00}=2v_{01}
\end{align*}
together with the relation
$$
u+v_{11}=2v_{00}.
$$

On the other hand, consider the 2-Veronese embedding of $\{x_0x_1x_2=y_0^4\}\subset \PP(1,1,2,1)$.
In coordinates $z_{02}=x_0^2$, $z_{20}=x_1^2$, $w=x_2$, $z_{00}=y_0^2$,
$z_{11}=x_0x_1$, $z_{01}=x_0y_0$, $z_{10}=x_1y_0$, and this hypersurface is given by the equation
$$
wz_{11}=z_{00}^2
$$
together with five 2-Veronese-type equations
\begin{align*}
	&z_{20}z_{02}=z_{11}^2, &z_{20}z_{01}=z_{10}z_{11},\\
	&z_{20}z_{00}=z_{10}^2, &z_{02}z_{10}=z_{01}z_{11},\\
	&z_{02}z_{00}=z_{01}^2.
\end{align*}
These correspond to the affine homogeneous relations above, so we can in fact realize our $Z$ as the hypersurface $\{x_0x_1x_2=y_0^4\}\subset \PP(1,1,2,1)$.
Thus, by degenerating the equation defining $X$, we get a degeneration of the del Pezzo surface of degree 2 to $Z$.
\end{ex}

\begin{remark}[{cf.~\cite[Remark~10]{Prz10}}]
In some cases (say, in the case of complete intersections in usual projective spaces or when all $m_i$'s are equal to 1) our very
weak Landau--Ginzburg models can be compactified in products of projective spaces (see the proof of Proposition~11 in~\cite{Prz09}). They are families of singular anticanonical hypersurfaces. The singularities of general members of these families are du Val along subspaces of codimension 2 and intersect transversally. Thus they have trivial canonical classes and crepant resolutions which means that they are birational to Calabi--Yau varieties. It follows that these very weak Landau--Ginzburg models are actually weak ones.
We also can prove this in other particular cases we are interested in. However we do not yet know how to prove this in the general case.
\end{remark}

\subsection{Degeneration via a Monomial Ideal}
Consider any projective variety $X\subset \PP^n$ defined by some homogeneous ideal $I\subset S=\CC[x_0,\ldots x_n]$. If $\prec$ is some monomial order for $S$, then there is a flat family degenerating $X$ to $X_\prec=V(\init_\prec (I))$, where $\init_\prec (I)$ is the initial ideal of $I$ with respect to the monomial order $\prec$. This is not of immediate help in finding toric degenerations of $X$, since in general, $X_\prec$ will be highly singular with multiple components and thus cannot be equal to or degenerate to a toric variety.

Instead, the point is to consider toric varieties embedded in $\PP^n$ which also degenerate to $X_\prec$.
Consider such a toric variety $Z$, and let $\mcH$ be the Hilbert scheme of subvarieties of $\PP^n$ with Hilbert polynomial equal to that of $X$. If $X$ corresponds to a sufficiently general point of a component of $\mcH$ and $X_\prec$ lies only on this component, then $X$ must degenerate to $Z$. This is the geometric background for the following theorem; the triangulations which appear correspond to degenerations of toric varieties to certain special monomial ideals with unobstructed deformations.

\begin{thm}[{\cite[Corollary 3.4]{ilten:11a}}]
Consider a three-dimensional reflexive polytope $\nabla$ with $m$ lattice points, $7\leq m \leq 11$, which admits a regular unimodular triangulation with the origin contained in every full-dimensional simplex, and every other vertex having valency $5$ or $6$. Then the Fano threefold $X_{2(m-3)}$ admits a degeneration to $\widetilde\PP(\nabla)=\PP(\nabla)$.
\end{thm}

\begin{ex}[$X_{12}$]\label{ex:mon12}
Consider the Laurent polynomial $f$ from Table \ref{table} for the Fano threefold $X_{12}$.  The dual of the Newton polytope $\nabla=\Delta_f^*$ is the convex hull of the vectors $\pm e_1$, $\pm e_2$, $e_3$, $-e_1-e_2$, $e_2+e_3$, and $-e_1-e_2-e_3$, see Figure \ref{fig:v12poly}. $\nabla$ has only one non-simplicial facet, a parallelogram. Subdividing this facet by either one of its diagonals gives a triangulation of $\partial \nabla$, which naturally induces a triangulation of $\nabla$ with the origin contained in every full-dimensional simplex.  It is not difficult to check that this triangulation is in fact regular and unimodular; furthermore, all vertices (with the exception of the origin) have valency $5$ or $6$. Thus, by the above theorem, $X_{12}$ degenerates to $\widetilde\PP(\nabla)$.
\end{ex}
\begin{figure}
\begin{center}
\includegraphics{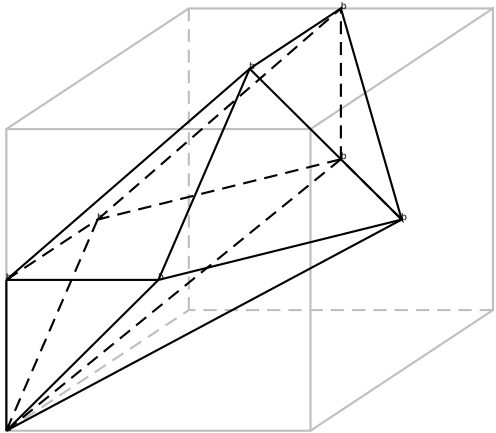}
\end{center}
\caption{$\Delta_f^*$ for $X_{12}$}\label{fig:v12poly}
\end{figure}

\begin{ex}[$X_{8}$, $X_{10}$, $X_{14}$, and $X_{16}$]\label{ex:mongen}
Consider the Laurent polynomial $f$ from Table \ref{table} for $X_d$, $d\in\{8,10,14,16\}$.  Similar to the above example for $d=12$,
one can check, either by hand or with the computer program TOPCOM \cite{TOPCOM}, that the polytope
$\Delta_f^*$ satisfies the conditions of the above theorem. Thus, there is a degeneration of
$X_d$ to the toric variety $\widetilde{\PP}(\Delta_f^*)$ corresponding to the Landau--Ginzburg
model given by $f$.
\end{ex}
\begin{ex}[$X_{18}$]\label{ex:v18}
	Consider the Laurent polynomial $f$ from Table \ref{table} for $X_{18}$. Here, $\nabla=\Delta_f^*$ has $12$ lattice points, so we cannot apply the above theorem, but similar techniques may be used to show the existence of the desired degeneration of $X_{18}$. Indeed, the dimension of the component $U$ corresponding to $X_{18}$ in the Hilbert scheme $\mathcal{H}_{X_{18}}$ of its anticanonical embedding is $153$, see \cite[Proposition 3.1]{ilten:12a}. The variety $Z=\widetilde{\PP}(\Delta_f^*)$ corresponds to a point $[Z]$ in $\mathcal{H}_{X_{18}}$ since its Hilbert polynomial agrees with that of $X_{18}$. A standard deformation-theoretic calculation using \cite{ilten:11d} shows that $[Z]$ is a smooth point on a component of dimension 153. It remains to be shown that this component is in fact $U$.

	Now, the boundary of $\nabla$ admits a triangulation such that one vertex has valency $6$, and every other vertex has valency $4$ or $5$. This triangulation is in fact induced by a regular unimodular triangulation of $\nabla$. In any case, $Z$ degenerates to the Stanley--Reisner scheme $Y$ corresponding to this triangulation, and $X_{18}$ does as well, see \cite[Corollary 3.3]{ilten:11a}. Furthermore, a standard deformation-theoretic calculation using \cite{ilten:11d} shows that at the point $[Y]$,
	$\mathcal{H}_{X_{18}}$ only has one $153$-dimensional component. Thus, $[Z]$ must lie on $U$, and $X_{18}$ must degenerate to $Z$.
\end{ex}

\section{The First Main Theorem}\label{sec:main}
We restate our first main theorem from the introduction:
\begin{thm}[First Main Theorem]
\label{theorem:Main}
Each smooth Fano threefold of Picard rank 1 has a toric weak Landau--Ginzburg model. 
More precisely, the Laurent polynomials in Table \ref{table} are weak Landau--Ginzburg models for corresponding Fano varieties, and, for each polynomial $f$ in the table, the corresponding Fano degenerates to $\widetilde\PP(\Delta_f^*)$.
\end{thm}
\begin{proof}
The existence of the toric degenerations follows from the methods of the previous section, and previously known small toric degenerations, \cite{galkin:08a}. Numbers 1--4, 11--14, and 16 follow from Theorem \ref{thm:ci}. Indeed, recall that the double solids are hypersurfaces in weighted projective spaces:
number 1 in $\PP(1,1,1,1,3)$ of degree 6, number 11 in  $\PP(1,1,1,2,3)$ of degree 6, and number 12 in $\PP(1,1,1,1,2)$ of degree 4. Numbers 10 and 15 admit small toric degenerations. Numbers 5--8 are dealt with in Examples \ref{ex:mon12} and \ref{ex:mongen}. Number 9 is covered by Example \ref{ex:v18}. Finally, the Fano variety number 17 is already toric.

The statement that the Laurent polynomials appearing in  Table \ref{table} are weak Landau--Ginzburg models was already shown in \cite{Prz09}.
\end{proof}

\section{Geometry of Compactified Fibers of the Landau--Ginzburg Potentials}\label{sec:picrk}
Mirror symmetry predicts more about the fibers of a Landau--Ginzburg potential than the fact that they compactify to Calabi--Yau varieties.  In particular, for the cases studied in this paper, the Picard lattices of the compactified fibers should have rank 19.  In the following, we verify this claim:

\begin{thm}[Second Main Theorem] \label{appendixthm} Let $X$ be a Fano threefold of Picard number one, and $f$ the Laurent polynomial for $X$ in Table 1.  Then the fibers of $f$ compactify to a family of K3 surfaces of Picard rank 19.  \end{thm}

At present, we know of no systematic proof of this theorem.  Instead, the proof will be done case by case in Section \ref{cases}.

\begin{remark}
\label{remark:uniqueness}
The rank of the Picard lattice is an important but rather rough invariant.  Actually computing the Picard lattices in each case is beyond the scope of this appendix but will give even more confirmation that the Landau--Ginzburg models given in this paper are correct mirrors.  \cite{Dol} gives a prescription for mirror symmetry for families of lattice-polarized K3 surfaces.  Anticanonical K3 surfaces in a Fano variety $X$ carry a natural lattice polarization induced from the polarization of $X$.  One would expect the generic fiber of a Landau--Ginzburg model for $X$ to be have the mirror-lattice polarization to the anticanonical family of $X$.  A forthcoming paper \cite{DKLP} will verify this expectation explicitly.  Note that the moduli space of K3 surfaces with a lattice polarization by a fixed rank 19 lattice is one-dimensional.  Hence all Landau--Ginzburg models for $X$ with the same lattice polarization will be birational (differ by flops).

Computing the Picard lattices will also show that the fiber of the Landau--Ginzburg models carry Shioda--Inose structures (see e.g. \cite{CD1}, \cite{CDLW}), which gives an explicit geometry correspondence between the K3 surfaces and product of elliptic curves with an isogeny.  This correspondence gives an explanation of the relationship observed in  \cite{Golyshev2005} that the quantum $ \scr D$-modules for Fano threefolds of Picard number one are related to modular forms.  \end{remark}

\subsection{Notation and Background}\label{sec:back}
As before, $N$ and $M$ will denote two dual lattices, where we now concentrate on the case of rank three.  Let $ f_i $ denote the Laurent polynomial defining the Landau--Ginzburg model in row $i$ of Table 1,  $ \Delta_i^* \subset  M_\bb{Q} $ its Newton polytope, and $ \Delta_i \subset  N_\bb{Q} $ its dual polytope.  Note that the roles of $M$ and $N$ have reversed from earlier in the paper, where Newton polygons were taken in $ N_{\bb Q} $.  This change is indicative of the fact that we are now working on the other side of mirror symmetry ($B$-model as opposed to $ A$-model).

We will write $ \langle r \rangle $ for a one-dimensional lattice generated by an element of square $r$.  $ A_n, D_n, E_n $ will denote the {\em negative-definite} root lattices of the corresponding Dynkin diagrams. $U$ will denote the rank-two hyperbolic lattice with intersection matrix
$$ \left( \begin{array}{cc} 0 & 1 \\ 1 & 0 \end{array} \right) .$$

We will use $ [x,y,z,w] $ as homogeneous coordinates on $ \bb P^3 $.  For distinct, non-empty subsets $ I, J, K \subset \{ 1, 2, 3, 4 \} $, we will write $H_I $ for the hyperplane defined by setting the sum of coordinates in $I$ equal to zero --- thus, for example, $H_{\{ 1\}} $ is the coordinate hyperplane $ x = 0 $, while $H_{\{2,4\}} $ is the hyperplane defined by $ y+w=0$.  We write $ L_{I,J} = H_I \cap H_J $, and $ p_{I,J,K} = H_I \cap H_J \cap H_K $.

In many cases, we will compactify the fibers of $ f_i $ to quartics in $ \bb P^3 $ with only ordinary double point singularities.  In those cases, we will identify some curves on the minimal resolutions of these singular quartics (which will be K3 surfaces) and give a heuristic argument for why the curves identified generate a lattice of rank 19.  When the exceptional locus consists of 18 curves, this heuristic argument is actually valid; in other cases, the actual proof consists of blowing up the singularities to compute the intersection matrix of the identified curves, then checking that this matrix has rank 19.  In the interest of not boring the reader to death, we will omit the details of these computations.  In other cases, we will use an elliptic fibration as described below.

\begin{defn} An {\em elliptic K3 surface with section} is a triple $ (X, \pi, \sigma) $ where $X$ is a K3 surface, and $ \pi: X \to \bb{P}^1 $ and $ \sigma: \bb{P}^1 \to X $ are morphisms with the generic fiber of $ \pi $ an elliptic curve and $ \pi \circ \sigma = \mathrm{id}_{\bb P^1} $. \end{defn}

Any elliptic curve over the complex numbers can be realized as a smooth cubic curve in $ \bb{P}^2 $ in {\em Weierstrass normal form}
\begin{equation} \label{Weierstrass} y^2 z = 4x^3 -g_2 x z^2 - g_3 z^3 \end{equation}
Conversely, the equation \eqref{Weierstrass} defines a smooth elliptic curve provided $ \Delta = g_2^3 -27g_3^2 \neq 0 $.

Similarly, an elliptic K3 surface with section can be embedded into the $ \bb{P}^2 $ bundle $ \bb{P}(\scr O_{\bb P^1} \oplus \scr O_{\bb P^1}(4) \oplus \scr O_{\bb P^1}(6)) $ as a subvariety defined by \eqref{Weierstrass}, where now $ g_2, g_3 $ are global sections of $ \scr O_{\bb{P}^1}(8) $, $ \scr O_{\bb P^1}(12) $ respectively (i.e. they are homogeneous polynomials of degrees 8 and 12).  The singular fibers of $ \pi $ are the roots of the degree 24 homogeneous polynomial $ \Delta = g_2^3 -27g_3^2 \in H^0(\scr O_{\bb P^1}(24)) $.  Tate's algorithm can be used to determine the type of singular fiber over a root $ p$ of $ \Delta $ from the orders of vanishing of $ g_2$, $g_3$, and $\Delta $ at $p$.

\begin{prop} \label{fibprop} \cite[Lemma 3.9]{CD1} A general fiber of $ \pi $ and the image of $ \sigma $ span a copy of $U$ in $ \mathrm{Pic}(X) $.  Further, the components of the singular fibers of $ \pi$ that do not intersect $ \sigma $ span a sublattice $ S$ of $ \mathrm{Pic}(X) $ orthogonal to this $U$,  and $ \mathrm{Pic}(X)/(U \oplus S) $ is isomorphic to the Mordell--Weil group $MW(X,\pi)$ of sections of $ \pi $.  \end{prop}

When K3 surfaces are realized as hypersurfaces in toric varieties, one can construct elliptic fibrations combinatorially.  As before, let $ \Delta \subset N_\bb{Q} $ be a reflexive polytope, and suppose $ P \subset N_{\bb Q} $ is a plane such that $ \Delta \cap P $ is a reflexive polygon $ \nabla $.  Let $ m \in M $ be a normal vector to $P$.  Then $P$ induces a torus-invariant rational map $ \pi_m:  \bb{P}(\Delta^*) \dashrightarrow \bb P^1 $  with generic fiber $ \bb P_{\nabla} $.  (This is just the Chow quotient of $ \bb{P}(\Delta^*) $ by the torus $ \bb{C}^* \otimes (m^\perp) $.) Restricting $ \pi_m $ to an anticanonical K3 surface and resolving indeterminacy, we get an elliptic fibration.  If $ \nabla  $ has an edge without interior points, this fibration will have a section as well.  See \cite[\S 3]{KreuzerSkarke} for more details.

\subsection{Picard Lattice Data for Fibers of the Landau--Ginzburg Models} \label{cases}
We now prove Theorem \ref{appendixthm} case-by-case, using one of four methods in each case:
\begin{enumerate}[{Method} 1:]
\item \label{quarticmethod} Compactify fibers of $ f_i$ to quartics with ordinary double points in $ \bb P^3$ and explicitly identify curves and singularities such that the strict transforms of the identified curves and the exceptional curves of the resolution of singularities generate a rank 19 lattice.
\item \label{linemethod} Compactify fibers of $ f_i $ to quartics in $ \bb P^3 $.  Identify a line $ \ell $ on the fibers.  Subtract $ \ell $ from the pencil of hyperplane sections containing $ \ell $ to obtain a pencil of plane cubics on the fibers.  Blowing up base points and resolving singularities gives an elliptic surface birational to the original fiber.  The pencils chosen in this paper will always have a base point, and an exceptional curve over a base point gives a section.
\item \label{productmethod} Compactify fibers of $ f_i $ in a product of weighted projective spaces and use an elliptic fibration given by an explicit map to $ \bb P^1 $.
\item \label{toricmethod} Compactify fibers of $ f_i $ in $ \bb P(\Delta^*_{f_i}) $ and specify a vector $ m$ that defines an elliptic fibration.
\end{enumerate}

\begin{enumerate}
\item
As shown in \cite[Remark 19]{Prz09}, this family compactifies to K3 surfaces mirror to $ \bb{WP}(1,1,1,3) $.  Explicitly, the form for the K3 fibers in \cite{Prz09} is
\begin{equation*} y_1 y_2 y_3 y_4^3 = \lambda, \: y_1+y_2 + y_3+ y_4 = 1 \end{equation*}
We make a different change of variable than the one that yields $ f_1 $, namely set $ x = y_1$, $ y = y_2 $, $ z = y_4 $.  Then the equation above reduces to
\begin{equation*} \tilde{f}_1 = x + y + z + \frac{\lambda}{xyz^3} - 1 = 0 \end{equation*}
We now use Method \ref{toricmethod} on $ \tilde{f}_1 $ with $ m= (1,0,1)$,  which gives a polarization by $ U \oplus E_7 \oplus D_{10} $.

\item
Using Method \ref{quarticmethod} gives a quartic with six $ A_3 $ singular points.  There are also lines $ L_{ \{ i \}, \{ 1,2,3,4 \}} $ for $ 1 \leq i \leq 4 $, each equal as a divisor to one-fourth hyperplane section.  Taking the minimal resolution of these quartics gives K3 surfaces, with the exceptional locus and the strict transform of one of these lines generating a rank 19 lattice in the Picard group.

Alternately, using Method \ref{linemethod} with $ \ell $ as any of the four lines above gives a polarization of the K3 surfaces by $ U \oplus E_6 \oplus A_{11} $.

\item
Compactify the fibers of $ f_3 $ as a family of anticanonical divisors in $ \bb{P}^1 \times \bb{P}^2 $ via $ (x,y,z) \mapsto ([x,1] \times [y,z,1])$. Explicitly, $ f_3^{-1}(\lambda) $ compactifies to the K3 surface
\begin{equation*} Y_\lambda = \{ ([x, x_0],[y,z,w]) \in \bb P^1 \times \bb P^2 \: | \:  (x+x_0)^2(y+z+w)^3 - \lambda x x_0 y z w = 0 \}. \end{equation*}
The projection $ \bb{P}^1 \times \bb{P}^2 \to \bb{P}^1 $ induces an elliptic fibration on $ Y_\lambda $ for generic $ \lambda $.  The map $ [x, x_0] \mapsto ( [x, x_0] , [1,-1,0]) $ gives a section of this elliptic fibration.  Putting the fiber over $ [1,a] $ into Weierstrass form
\begin{equation*} \frac{a^3 \lambda^3(24(1+a)^2-a\lambda)}{48} X -\frac{a^4 \lambda^4 (36 (1 + a)^2 (6 (1 + a)^2 - a s) + a^2 s^2)}{864}+X^3+Y^2 = 0 \end{equation*}
and using Tate's algorithm, we see singular fibers of Kodaira type $IV^* $ at $ a =0, \infty $; $ I_6 $ at $ a = -1 $; and $ I_1 $ where $ 27(a+1)^2 - \lambda a = 0 $.  Hence the rank 19 lattice $U \oplus E_6 \oplus E_6 \oplus A_5 $ embeds in the Picard lattice of $ Y_\lambda $.

\item Similar to the case above, we compactify as anticanonical K3 surfaces in  $ \bb P^1 \times \bb P^1 \times \bb P^1 $.  Projection onto one of the $ \bb P^1 $ factors gives the generic K3 fiber an elliptic fibration with section.  Putting this into Weierstrass form and running Tate's algorithm gives an embedding of the rank 19 lattice $ U \oplus A_7 \oplus D_5 \oplus D_5 $ into the Picard lattice of the generic fiber.

\item Using Method \ref{quarticmethod}, there are singularities at $ p_{ \{i \}, \{j \}, \{4\}}$ for $1 \leq i \neq j \leq 3$ of type $ D_4 $ and at $  p_{\{i \} \{ j \}, \{k,4\}}$ where $\{i,j,k\} = \{1,2,3\}$ of type $A_1$.  Thus the exceptional curves generate a sublattice of rank 15.  The quartics also contain lines $ L_{ \{i\}, \{j, 4\}} $ and conics $ C_{\{i,j,4\}} $ for $ 1 \leq i \neq j \leq 3 $, subject to relations  from
\begin{align*}
H_{\{1\}} & =  2 L_{\{1\} , \{2,4 \}} + 2 L_{\{1 \}, \{3,4 \}} \\
H_{\{2\}} &= 2 L_{\{2\},\{1,4\}} + 2 L_{\{2\},\{3,4\}} \\
H_{\{3\}} &= 2 L_{\{3\},\{1,4\}} + 2 L_{\{3\},\{2,4\}} \\
 H_{\{1,2,4\}} &= L_{\{1\},\{2,4\}} + L_{\{2\},\{1,4\}} + C_{\{1,2,4\}} \\
H_{\{1,3,4\}} &= L_{\{1\},\{3,4\}} + L_{\{3\},\{1,4\}} + C_{\{1,3,4\}} \\
H_{\{2,3,4\}} &= L_{\{2\},\{3,4\}} + L_{\{3\},\{24\}} + C_{\{2,3,4\}}
\end{align*}
which leave a lattice of rank 19.

Explicitly computing the intersection matrix for the curves identified verifies that they generate a lattice of rank 19.

Alternately, we may use Method \ref{linemethod} with $ \ell =  L_{ \{1\}, \{2,4\}} $.  The induced fibration has singular fibers of types $ I_2^* $, $I_1^* $, $ I_6 $, and $ 3 I_1 $.  It also has a section of infinite order and a 2-torsion section.  Hence the Picard lattice of the generic member of this family is a rank 19 lattice containing $ U \oplus D_6 \oplus D_5 \oplus A_5 $ with quotient $ \bb Z \oplus \bb Z_2 $.

\item Using Method \ref{quarticmethod}, there are $A_1 $ singularities at $ [1,-1,0,0]$, $[1,0,-1,0]$, and $ [0,1,-1,0]$; $ A_2$ singularities at $ [1,0,0,0] $ and $[0,0,1,-1]$; and $A_3$ singularities at $ [0,1,0,0] $ and $ [1,0,0,-1]$.  These quartics also contain twelve lines:
\begin{align*}  L_{ \{1\}, \{2,3\}},
 L_{ \{1\}, \{3,4\}},
 L_{ \{1\}, \{2,3,4\}},
 L_{ \{2\}, \{3\}},
 L_{ \{2\}, \{3,4\}},
 L_{ \{2\}, \{1,3,4\}}, \\
 L_{ \{3\}, \{4\}},
 L_{ \{3\}, \{1,4\}},
 L_{ \{3\}, \{1,2,4\}},
 L_{ \{4\}, \{1,3\}},
 L_{ \{4\}, \{2,3\}},
 L_{ \{4\}, \{1,2,3\}} \end{align*}
subject to relations coming from setting equal the hyperplane sections $ H_{\{1\}}$, $ H_{\{2\}}$, $ H_{\{3\}}$, $ H_{\{4\}}$, $ H_{\{1,3,4\}}$, $ H_{\{1,2,3,4\}}$, $ H_{\{3,4\}}$, and $ H_{\{2,3\}}$.  These relations show that only six of these twelve lines are linearly independent.  Hence the exceptional locus and strict transforms of lines generate a sublattice of the Picard lattice of the minimal resolution K3's of rank 13+6 =19.

By explicitly computing the intersection matrix for the 25 rational curves identified, we verify that they generate a rank 19 lattice.

\item Again, use Method \ref{quarticmethod}.  The quartics are defined by
\begin{equation*} (x+y+z+w)(yz(x+y+z+w)+(y+z+w)(z+w)^2) - \lambda xyzw=0 \end{equation*}
The singularities are: $A_1$ at $[0,1,0,-1]$;  $A_2 $ at $[1,0,0,0]$, $ [0,1,-1,0]$, and $[\lambda,0,-1,1]$; $ A_3$ at $[0,0,1,-1]$; and $ A_4$ at $[1,-1,0,0]$.  The quartics contain eight lines
\[ L_{\{i\}, \{1,2,3,4\}} \: (1 \leq i \leq 4), \: L_{\{2\},\{3,4\}}, \:  L_{\{3\}, \{ 2,4 \}}, L_{\{3\}, \{4\}}, L_{\{2,3,4\},*} = \{ y+z+w= x - \lambda w = 0 \}\]
and two conics
\[ C_1 = \{ x=yz+(z+w)^2 = 0 \}, \: C_4 = \{ w =xy+(y+z)^2=0 \} \]
subject to relations coming from setting equal the hyperplane sections $ H_{\{1\}}$, $ H_{\{2\}}$, $ H_{\{3\}}$, $ H_{\{4\}}$, $ H_{\{2,3,4\}} $, and $ H_{\{1,2,3,4\}}$.  These relations show that these ten rational curves on the quartic generate a sublattice of rank five in the Picard lattice.  Hence the exceptional locus and the strict transforms of these ten curves generate a rank 19 sublattice of the Picard lattice of the minimal resolution, as can be explicitly verified by computing the intersection matrix for the curves identified.

\item Compactifying to singular quartics gives singularities of type $ A_1 $ at the six points
\[ [-1,0,0,1], [0,-1,0,1], [0,0,-1,1], [1,-1,0,0],[1,0,-1,0],[0,1,-1,0] \]
and singularities of type $A_2$ at the three points $ [1,0,0,0], [0,1,0,0], [0,0,1,0] $.  There are also 13 lines
\[ L_{\{i\}, \{1,2,3,4\}}, L_{\{j\}, \{4\}}, L_{\{j\},\{k,4\}} \: \mathrm{for} \: 1 \leq i \leq 4, \:  1 \leq j \neq k \leq 3 \]
subject to relations from setting equal the hyperplane sections by $ H_{\{i \}}, H_{\{j,4\}}, H_{\{1,2,3,4\}}$ for $ 1 \leq i \leq 4 $, $ 1 \leq j \leq 3 $.  These relations show that the lattice generated by the 13 lines has rank 7.  Hence the strict transforms  of the lines and the exceptional locus generate a lattice of rank 19.

\item Using Method \ref{linemethod} with $ \ell = L_{ \{4\}, \{1,2,3\}} $ gives a polarization of the Picard lattice of the minimal resolution by the rank 19 lattice $ U \oplus A_8 \oplus A_2 \oplus A_1 \oplus E_6 $.

\item The quartic compatification contains lines
	\begin{align*}
	 L_{ \{1\}, \{3\}},
L_{ \{1\}, \{4 \} },
L_{ \{1\}, \{2,4\}},
L_{\{1\}, \{3,4\}},
L_{\{2\}, \{3\}},
L_{ \{2 \}, \{4\} },
L_{\{2\}, \{1,4\}},
L_{\{2\},\{3,4\}},
L_{\{3\}, \{1,4\}}, \\
 L_{\{3\}, \{2,4\}},
L_{ \{1,3 \}, \{4\} }
L_{ \{2,3 \}, \{4 \} },
L_{ \{1,4\}, *} = \{ x+w=(s-2)x+y =0 \},
\\
L_{\{2,4\},*} = \{ y+w = (s-2)y+x = 0 \}
\end{align*}
and conics
\begin{align*}
 C_{\{3,4\}} = \{ z+w = xy +(\lambda-2)z^2 =0 \},\\
 C_{\{1,2,4\}} =\{ x+y+w= x y + (\lambda - 3) (x + y) z + z^2=0\} , \\
  C = \{ z = (\lambda+1)w, (\lambda+1)w^2+xy=0 \},\\
  C' = \{ z = (\lambda+1)w, 2 w (w + x + y) + \lambda w (x + y) + x y=0 \}
\end{align*}
subject to relations coming from $ H_{\{i\}} $,
and singularities of types $ A_3 $ at $[1,0,0,0]$ and $ [0,1,0,0]$,  $ A_2 $ at $ [0,0,1,0] $, and $ A_1 $ at $ [-1,0,0,1]$ and $ [0,-1,0,1]$.

The lines are subject to relations from setting equal $ H_{\{1\}} $, $H_{\{2\}}$, $ H_{\{3\}} $, $H_{\{4\}} $, $ H_{\{1,3\}} $, $H_{\{2,3\}} $, $ H_{\{1,4\}} $, $ H_{\{2,4\}} $, and $ H_{\{3,4\}} $.

Computing the intersection matrix for these curves verifies that they generate a rank 19 lattice.

\item As shown in \cite{Prz09}, the fibers of $ f_{11} $ can be compactified to quartics
\begin{equation*} \tilde{f}_{11} =  x^4 - (\lambda y -z)(x w -xy-w^2)z = 0 \end{equation*}
These quartics contain lines
\[ L_{ \{1\}, \{3\}}, L_{\{1\}, \{4\}}, L_{\{1\},*} = \{ x= \lambda y - z =0\} \]
We now use Method \ref{linemethod} on the fibers of $ \tilde{f}_{11} $ with $ \ell =  L_{\{1\}, \{3\}} $.  Putting this fibration into Weierstrass form and applying Tate's algorithm gives a polarization by $ U \oplus E_7 \oplus D_{10} $.

\item Using Method \ref{linemethod} with $ \ell =  L_{\{1\}, \{2,4\}} $, the induced polarization is by the rank 19 lattice $ U \oplus E_6 \oplus A_{11} $.

\item Using Method \ref{linemethod} with $ \ell =  L_{\{1\}, \{4\}} $ gives an elliptic fibration that results in a polarization by $ U \oplus E_6 \oplus E_6 \oplus A_5 $.

\item Using Method \ref{toricmethod} with $ m=(0,0,1) $ yields a fibration with fibers of type $ I_8 $ at $ \infty $ and $ I_1^* $ at $ t = \frac{1}{2} \left(\lambda \pm \sqrt{\lambda ^2+16}\right) $.  Hence the fibers carry an $ U \oplus A_7 \oplus D_5 \oplus D_5 $ polarization, as in case number 4 above.

\item

Using Method \ref{toricmethod} with $ m = (1,1,0) $ induces an elliptic fibration with Weierstrass form
\begin{equation*}
-\frac{1}{48} t^2 P(s,t) u+\frac{1}{864} t^3 \left(s^2 (-t)+4 t^2+12 t+8\right) \left( P(s,t) + 24(1+t)^2 \right)+u^3+v^2= 0
\end{equation*}
where $ P(s,t) = s^4 t^2-8 s^2 t^3-24 s^2 t^2-16 s^2 t+16 t^4+24 t^3-8 t^2-24 t-8 $.  This fibration has a section of infinite order given by
\begin{equation*}
t \mapsto \left( -\frac{1}{12} t \left(s^2 t+8 t^2+12 t+4\right),-\frac{1}{2} s t^2 (t+1)^2 \right) = (u,v)
\end{equation*}
and a 2-torsion section given by
\begin{equation*}
t \mapsto \left( \frac{1}{12} \left(-s^2 t+4 t^2+12 t+8\right), 0 \right) = (u,v). \end{equation*}

Hence by Proposition \ref{fibprop}, $ NS(X) $ is a rank 19 lattice containing $ U \oplus D_6 \oplus D_5 \oplus A_5 $ with quotient $ \bb Z \oplus \bb Z_2 $.

\item Using Method \ref{toricmethod} with $ m = (1,0,0) $ gives a fibration with lattice $ U \oplus E_6 \oplus E_6 \oplus A_5 $ plus additional sections.

\item Using Method \ref{toricmethod} with $m=(0,0,1)$ yields a polarization by $ U \oplus E_6 \oplus A_{11} $.

\end{enumerate}

\begin{remark} \cite{Golyshev2005} shows that the Landau--Ginzburg models for these cases have the same variation of Hodge structure (up to pullback) as a modular variation associated to products of elliptic curves with isogeny.  Explicitly, for $X$ one of the Fano threefolds under consideration, let $ (N,d) = \left(\frac{\deg(X)}{2 \cdot \mathrm{ind}(X)^2}, \mathrm{ind}(X)\right)$.  Let $ X_0(N)+N $ denote the modular curve $ \overline{(\Gamma_0(N)+N) \backslash \bb{H}} $, and let $ t_N $ be a hauptmodul for $ X_0(N)+N $ such that $t_N = 0 $ at the image of the cusp $ i \infty $.  The Picard--Fuchs equation for the Landau--Ginzburg model of $X$ is now the pullback of the symmetric square of the uniformizing differential equation for $ X_0(N)+N $ by $ \lambda = t_N^d $.

We can check that the pullback part of Golyshev's theorem follows in a straightforward way from the geometry of the fibers of the Landau--Ginzburg model:

\begin{itemize}

\item {\bf Cases 1 and 11:} Both have polarizations by $ U \oplus E_7 \oplus D_{10} $.  Clearly, since the moduli space of $ U \oplus E_7  \oplus D_{10} $ polarized  K3 surfaces is 1-dimensional, we see {\em a posteriori} that the Landau--Ginzburg models $ f_1 $, $ \tilde{f}_{11} $ have isomorphic K3-compactified fibers.

\item {\bf Cases 2, 12, and 17:}    Similarly, since the moduli space of $ U \oplus E_6 \oplus A_{11} $ polarized  K3 surfaces is 1-dimensional, we see {\em a posteriori} that the Landau--Ginzburg models $ f_2 $, $ f_{12} $, and $ f_{17} $ have isomorphic fibers.    Writing the Weierstrass forms for the elliptic fibrations that give this polarization in each case, we can match the fibrations fiberwise to check that indeed case 12 is a pullback of case 2 by $ \lambda \mapsto \lambda^2 $, and similarly case 17 is a pullback of case 2 by $ \lambda \mapsto \lambda^4 $.
\item {\bf Cases 3, 13, and 16:} Similar to the previous cases, using the polarizations by $ U \oplus E_6 \oplus E_6 \oplus A_5 $.
\item {\bf Cases 4 and 14:} Similar to the previous cases, using the polarizations by  $ U \oplus A_7 \oplus D_5 \oplus D_5 $.
\item {\bf Cases 5 and 15:}  Similar to the previous cases, using the elliptic fibrations with Mordell--Weil group $ \bb Z \oplus \bb Z_2 $.
\end{itemize}

\end{remark}

\bibliography{fanodegen}

\address{
Nathan Owen Ilten\\
Dept. of Mathematics\\
University of California\\
Berkeley, CA 94720
}{nilten@math.berkeley.edu}

\address{
Jacob Lewis \\
Fakult\"at f\"ur Mathematik\\
Universit\"at Wien \\
Garnisongasse 3/14 \\
A-1090 Wien,
Austria
}{JacobML@uw.edu}

\address{
Victor Przyjalkowski\\
Steklov Mathematical Institute\\
Gubkina st., 8\\
119991, Moscow, Russia
}{victorprz@mi.ras.ru,\\
victorprz@gmail.com}

\end{document}